\documentclass[11pt,twoside]{article}
\usepackage[english]{babel}
\usepackage{amssymb}
\usepackage[mathscr]{eucal}
\usepackage{amsmath,amsthm}
\usepackage{mathrsfs}
\usepackage{lgreek}
\usepackage[shortlabels]{enumitem}
\usepackage[colorlinks=true,
   urlcolor=blue,           filecolor=green,      
   citecolor=green,      
   linkcolor=red,           bookmarks=true,
  unicode,
   plainpages=false,   ]{hyperref}
\usepackage{color}

\definecolor{royalblue}{rgb}{0,0,0.128}
\def\bp{\begin{proof}}
\def\ep{\end{proof}}
\def\n{\nabla}

\def\sfrac#1#2{\mbox{\Large$\frac{#1}{#2}$}}
\def\intl#1{\int\limits_{#1}}

\def\dm {|\hskip-0.05cm|}

\def\displ{\displaystyle}

\def\VS{\vspace{6pt}\\\displ }

\def\rf#1{{\rm(\ref{#1})}}

\newcommand{\essinf}{\mathop{\mathrm{ess\,inf}}}
\newcommand{\R}{\mathbb{R}}
\newcommand{\N}{\mathbb{N}}

\def\à{à}

\def\vep{\varepsilon}

\def\be{\begin{equation}}
\def\ba{\begin{array}}
\def\ea{\end{array}}
\def\ee{\end{equation}}
\def\vs1{\vspace{1ex}}
\def\vp{\varphi}
\def\ov{\overline}

\def\Ã©{\'{e}}
\def\Ãš{\`{e}}
\newtheorem{lemma}
{\bf Lemma} 
\font\sc=cmcsc10
\title{\Large\bf Weighted estimates for the Stokes semigroup in the half-space}
\author{\sc  Angelica Pia Di Feola and Vittorio Pane\thanks{Dipartimento di Matematica e Fisica,  
Universit\`{a} degli 
Studi della Campania
``L. Vanvitelli'', via Vivaldi 43, 81100 \null\hskip0.55cmCaserta,
 Italy.\newline\null\hskip0.55cm
vittorio.pane2@unicampania.it
 \newline \null\hskip0.55cm 
angelicapia.difeola@unicampania.it}}
\date{\today}
\begin{document}
\markboth{\footnotesize\rm  A. P. Di Feola, V. Pane}{\footnotesize\rm
Weigthed estimates for the Stokes semigroup in the half-space}
\maketitle
\noindent
\newcommand{\red}{\protect\bf}
\renewcommand\refname{\centerline
{\red {\normalsize \bf References}}}
\newtheorem{ass}
{\bf Assumption} 
\newtheorem{defi}
{\bf Definition} 
\newtheorem{theorem}
{\bf Theorem} 
\newtheorem{rem}
{\sc Remark} 
\newtheorem{coro}
{\bf Corollary} 
\newtheorem{prop}
{\bf Proposition} 
\renewcommand{\theequation}{\arabic{equation}}
\setcounter{section}{0}
\section{Introduction}

In this note, we study the initial-boundary value problem for the Stokes system in the half-space within the weighted Lebesgue spaces setting. The leading idea of this work is to provide  a first step toward the study of solutions to the Navier–Stokes system in weighted spaces and their corresponding asymptotic properties related to the spatial and time variables, both in the half-space and in exterior domains. Indeed, our future aim is to extend the results contained in \cite{GM}, in which the authors study the perturbations of the steady case, and in \cite{MP} where, with respect to the Navier-Stokes Cauchy problem, the asymptotic behavior of the solution is established. We point out that in both of the quoted papers, the scaling invariant Lebesgue weighted space $L^p_\alpha $, with a radial weight $|x|^{\alpha}$, $\alpha=1-\frac{n}{p}$, $p>n\geq 3$, is considered.\par
We consider a new scaling invariant space, say $L^p_w$, where the weight function is defined in \eqref{wf}. In this new framework, we generalize the result concerning the existence of a unique strong solution to the Stokes Cauchy problem. This generalization allows us to establish an existence and uniqueness result for regular solutions to the initial boundary value problem (IBVP) to the Stokes system within our functional framework.\par
There is a wide literature concerning the well posedness in the setting of weighted spaces. We quote some contributions that for different aspects seem close to our result. In \cite{KK1}, the authors consider a weight function in the form $\ov{w}(x)=<x>^{sp}=(1+|x|^2)^\frac{sp}{2}$ with $1<p<\infty$, $0\leq s<(n-1)(1-\frac{1}{p})$. In particular, they provide a result on the  weighted $L^q$-$L^p$ estimates, $1<p\leq q<\infty$, for the Stokes semigroup in the half-space and in perturbed half-space (with $C^2$ boundary) for an initial datum belonging to $L^p_{\ov{w}}\cap L^p$. In \cite{farwingarx}, in another context,  the authors consider the same weight function providing $L^q$-$L^p$ estimates on the Stokes semigroup in exterior domains.\\
In \cite{han},
weighted estimates for the Stokes flow in the half space $\R^n_+$ are provided: the author considers three types of weight function, in particular the radial one.\par
As will be clear, our result is coherent with estimates in  \cite{han}, relative to the scaling invariant setting.\\
The literature also contains contributions devoted to the weighted space theory for the Stokes resolvent system. In this regard, we refer to \cite{farsohr, Fro, FroII}.\par
In the following, the symbol $\mathbb{R}^n_{+}$ denotes the $n$-dimensional half space, i.e. 
$$\mathbb{R}^n_{+}=\{x\in \mathbb{R}^n\,:\, x_n>0\}. $$
Let $p>n$, and $\alpha=1-\frac{n}{p}$. Moreover, for $m\in \N$, let $\{\alpha_j\}_{j=1,\dots,m}$ be such that $\alpha_j\geq 0$ for all $j\in\{1,\dots,m\}$ and $\displaystyle \sum_{j=1}^m\alpha_j=\alpha$. By $w(x)$ we denote the following weight function
\be \label{wf}
w(x)=\prod_{j=1}^m|x-\bar{x}_j|^{\alpha_j}\,,
\ee
where $\bar{x}_j$, for $j=1,\dots,m$, are fixed (not necessarily distinct) points in $\R^n_+$.\\
Let  $\Omega$ denote either $\R^n$ or $\R^n_+$. We consider the following weighted Lebesgue space:
\begin{equation}\label{DSP}
L^p_{w}(\Omega):=\{u\, :\, w(x)\,u \in L^p(\Omega)\}\,,
\end{equation}
endowed with the natural norm 
$$\dm   u\dm _{L^p_{w}(\Omega)}:=\dm u\,w\dm _{L^p(\Omega)}\,.$$
By $p'$ we mean the conjugate exponent of $p$. We set $\alpha'_j:=-\alpha_j$,  $\alpha'=-\alpha$, and
\begin{equation}\label{dwf}
    w'(x)=\prod_{j=1}^m|x-\bar{x}_j|^{\alpha'_j}\,.
\end{equation} 
By the symbol $L^{p'}_{w'}(\Omega)$, we mean the dual space of $L^p_w(\Omega)$\footnote{The reader is referred to \cite{Kuf} for further details}. Moreover, we set
\[\mathscr{C}_0(\Omega):=\{\phi \in \mathit{C}_0^{\infty}(\Omega)\, :\, \n \cdot \phi=0\}\,,\]
and by $J^{p}_{w}(\Omega)$ we denote the completion of $\mathscr{C}_0(\Omega)$ in the $L^p_{w}(\Omega)$-norm\footnote{For more details, see Theorem\,\ref{HWD} and appendix,  and  \cite{Fr-UF,HDS} too.}.\\
Where there is no ambiguity, we will indicate $\dm   \cdot \dm _{L^p_{w}(\mathbb{R}^n)} = \dm   \cdot \dm _{w,p} $ and $\dm   \cdot \dm _{L^p(\mathbb{R}^n)} = \dm   \cdot \dm _{p}$. \par
Hereafter we use the symbol $D_x^{\mathtt{k}}$ to denote the partial derivatives with respect to the variable $x \in \R^n$, where $\mathtt{k}$ is a multi-index with $|\mathtt{k}|=k$. In addition, we use the symbol $D_t^h$ to denote  the $h$-th partial derivative with respect to the real variable $t$.\par
We consider the following initial boundary value problem for the Stokes system
\begin{equation}\label{problem}
\begin{array}{ll}
v_t + \nabla \pi= \Delta v \,,\;&\text{in}\,\, (0,\infty) \times \mathbb{R}^n_{+}\,, \\
\nabla \cdot v=0\,, &\text{in}\,\, (0,\infty) \times \mathbb{R}^n_{+}, \\
v_{| x_n=0}\,=0 \,,\\
v(0,x)=v_0(x)\,, 
\end{array}
\end{equation}
under the assumptions that $v_0 \in L^p_{w}(\mathbb{R}^n_{+})$, and $\nabla \cdot v_0=0$ in weak sense.\par
We are interested in proving the following Theorem:
\begin{theorem}\label{Theoremprincipale}
\sl
Let  consider $p>n\geq 3$. Let $w$ be the weight function defined in \eqref{wf}. Then, for all $v_0\in J^p_w(\R^n_+)$ there exists a unique smooth solution to the problem \eqref{problem}. Moreover, there exists a constant $c$ independent of $v_0$ and of $s\geq 0$, such that
\be \label{smgP}
\dm v(t)\dm_{L^p_w(\R^n_+)}\leq c \dm v(s)\dm_{L^p_w(\R^n_+)}, \quad \mbox{for all } t>s\geq 0, 
\ee
for any $l, k \in \mathbb{N}_0$, $q>n$ and for all $t>0$, the following estimate holds
\be \label{RP}
\dm D_t^l \nabla^k v(t)\dm_{L^q(\R^n_+)}
\leq c\, (t-s)^{-\frac{n}{2}(\frac{1}{n}-\frac{1}{q})-\frac{k}{2}-l}\dm v(s)\dm_{L^p_w(\R^n_+)},  \quad \mbox{for } t>s\geq 0\,,
\ee
and the pressure term fulfills the following estimate
\begin{equation}\label{pressure}
 \dm  \nabla \pi_v(t)\dm_{L^q(\R^n_+)}\leq c\, (t-s)^{-\frac{n}{2}(\frac{1}{n}-\frac{1}{q})-1}\dm v(s)\dm_{L^p_w(\R^n_+)}\,, \quad \mbox{for } t>s\geq 0\,.  
\end{equation}
Moreover, the solution admits the integral representation \eqref{IRv} and satisfies the following initial condition
\be \label{LP}
\lim_{t \to 0^+} \dm v(t)-v_0 \dm_{L^p_w(\R^n_+)}=0.
\ee
Finally, for any $q>n$, $\varepsilon>0$, and $T>0$ we have
\be \label{regular_Hs}
v\in C([0,T);J^p_w(\R^n_+))\cap C((\varepsilon,T); L^q(\R^n_+))\,.
\ee
\end{theorem}
We point out that, if one considers $\ov{x}_1=\ov{x}_2=\dots=\ov{x}_m=0$, the result agrees with that obtained for the radial weight in \cite{han}.\par
The plan of the paper is the following. In Section \ref{preliminary}, we introduce some notations and definitions, some properties of the weight function \eqref{wf}, and we provide an interpolation inequality. In Section \ref{cauchy oroblem} we consider the Stokes Cauchy problem deriving useful properties that allow us to prove Theorem \ref{Theoremprincipale} in Section \ref{IBVP problem}.

\section{Preliminary results and notations}\label{preliminary}
In this section, we recall some notions that will be employed in our proofs and we prove some preliminary results.\par
Firstly, we introduce the Lorentz spaces\footnote{\,For a wider background, we refer
the reader {\it e.g.} to \cite{Hunt}, \cite{PKJF-LS} .}. Let $\Omega \subseteq \R^n$, and $g(x)$ be a (Lebesgue) measurable function, we denote by $\mu_g$ the distribution function of $g$, that is:
\[\mu_g(\lambda):=meas(\{x\in \Omega\,:\, |g(x)|>\lambda\})\,,\quad \text{for all}\,\, \lambda>0,\]
and we define the non-increasing rearrangement of $g$ as:
\[g^*(t):=\inf\{\lambda\,:\, \mu_g(\lambda)\leq t\}\,, \quad \text{for}\,\, t\in[0,\infty).\]
Now, assume that $0<p\leq q \leq \infty$. The Lorentz space $L^{p,q}(\Omega)$ is the collection of all measurable functions $g(x)$ such that $\dm g  \dm _{L^{p,q}(\Omega)}<\infty$, where:
\begin{equation}
\dm g  \dm _{L^{p,q}(\Omega)}=
\begin{cases}
\bigg(\mbox{\Large$ \int$}_0^{\infty}[t^{\frac{1}{p}}g^*(t)]^q\,  \frac{dt}{t}\bigg)^{\frac{1}{q}}\quad &\text{if}\,\, 0<q<\infty,\\
\displaystyle\sup_{t\in (0,\infty)} t^{\frac{1}{p}}g^*(t) \quad &\text{if}\,\, q=\infty.
\end{cases}
\end{equation}
Additionally, we consider the Helmholtz decomposition for the space $L^q_{w}(\R^n_+)$, $q\in(1,\infty)$. Precisely, we set
\begin{gather*}
J^q_{w}(\mathbb{R}_+^n)=\{v\in L^q_{w}(\mathbb{R}_+^n)\, |\, (v,\nabla h)=0 \,\, \text{for all}\, h \in W^{1,q'}_{loc}(\mathbb{R}_+^n,w')\text{,}\, \nabla h \in L^{q'}_{w'}(\R_+^n)\}\,,\\
G^q_{w}(\mathbb{R}_+^n)=\{u\in L^q_{w}(\mathbb{R}_+^n)\, |\, \exists \pi\, : \,\, u=\nabla \pi \text{,}\,\text{with} \, \pi \in W^{1,q}_{loc}(\mathbb{R}_+^n,w)\text{,}\, \nabla \pi \in L^{q}_{w}(\R_+^n)\}\,,
\end{gather*}
where by the symbol $W^{1,q}_{loc}(\mathbb{R}_+^n,w)$ we mean the space of functions $\{\pi : \pi \in L^1_{loc}(\R^n_{+}) \,\text{and}\, \nabla \pi \in L^q_w(\R^n_+)\}$.\\ 
We have the following result:
\begin{theorem}\label{HWD}{\sl 
Let $q \in (1,+\infty)$. Then we get 
\begin{equation*}
L^q_{w}(\R_+^n)=J^q_{w}(\mathbb{R}_+^n)\oplus G^q_{w}(\mathbb{R}_+^n)\,,
\end{equation*}
 that is for all $u \in L^q_{w}(\R_+^n)$, $u=v+\nabla \pi_u$, with the following integral identities:
\begin{gather*}
(u,\nabla \pi)=(\nabla \pi_u ,\nabla \pi) \quad \text{for all}\, \, \nabla \pi \in G_{w'}^{q'}(\R_+^n) ,\\
(v,\nabla \pi)=0 \quad \text{for all}\, \, \nabla \pi \in G_{w'}^{q'}(\R_+^n),
\end{gather*}
and
\begin{equation*}
\dm v\dm _{w, q} + \dm  \nabla \pi \dm _{w, q} \leq C \dm u\dm _{w, q}\,,
\end{equation*}
with $C$ independent of $u$.}
\end{theorem}
\begin{proof}
For the sake of completeness, we give the proof in the Appendix, Theorem \ref{dec_helmotz_teorema}. However, we also quote {\it e.g.}, \cite{Fr-UF,HDS}.
\end{proof}
In the following, we denote by $\mathcal{E}(x) $ the fundamental solution of the Laplace equation, i.e. for $n\geq 3$
\begin{equation}\label{Erapp}
    \mathcal{E}(x) = \frac{1}{n(n-2)V(n) |x|^{2-n}} 
\end{equation}
where $V(n)$ denotes the volume of the unit sphere in $\R^n$, and as $H(t,x)$ the fundamental solution of the heat equation, that is,
\begin{equation}\label{Hrapp}
    H(t,x) = \begin{cases}
\frac{1}{(4\pi t)^{\frac{n}{2}}}e^{-\frac{|x|^2}{4t}} & \text{if } \; t>0\\
0 & \text{if } \; t=0.
\end{cases}
\end{equation}
Now, we introduce the definition of a particular class of weight functions 
\begin{defi}{(Muckenhoupt $\text{A}_\text{p}$-weight)}
\sl{A positive function $w\in L^1_{loc}(\R^n)$ is called $\text{A}_\text{p}$-weight, written $w\in \text{A}_\text{p} $, if, and only if, there exists a constant $C>0$ such that for $p=1$,
$$\int_Q w(x) \, dx \leq C (\essinf_Qw)\left|Q\right|, $$
and, if $1<p<\infty$,
$$\left(\frac{1}{\left|Q\right|} \int_Q w \, dx\right)\left(\frac{1}{\left|Q\right|} \int_Q w^{\frac{1}{p-1}} \, dx\right)^\frac{1}{p-1} \leq C.$$
for every cube $Q\subset \R^n$.}
\end{defi}
In particular, our weight function $w(x)$ defined in \eqref{wf} belongs to $A_p$ class. For the weighted theory of singular integrals we refer the reader to \cite{CF}, \cite{Garcia}.\par In what follows, we will make use of the following results
\begin{theorem}
    \sl{ Let $S$ be a regular singular integral operator, $p \in (1,\infty)$, and $w\in A_p$. Then $S$ is bounded in $L^p_w(\R^n)$. More precisely, there is a constant $C$ depending only on $S$, $w$, and $p$ such that
\begin{equation}\label{WCZ}
    \dm Sf\dm_{w,p}\leq C \dm f\dm_{w,p}\, ,
\end{equation} 
    for all $f\in L^p_w(\R^n)$.
    }
\end{theorem}
\begin{proof}
    For the proof, see Chapter IV, Theorem 3.1 in \cite{Garcia}.
\end{proof}
\begin{lemma}\label{lapes}
\sl{    Let consider $w(x)$ as in \eqref{wf}. Then the following there holds
    \begin{equation}\label{lapneg}
        \Delta w'^{\,p'}(x) \leq (\alpha' p' + n -2) \sum_{h=1}^m (\alpha'_hp')
\prod_{i=1}^m |x-\ov{x}_i|^{\alpha'_ip'-2\delta_{hi}}< 0\,,    \end{equation}
where $\alpha\,' $, $w'(x)$  are defined in \eqref{dwf}, $p'$ is the conjugate exponent of $p>n$, and $\delta$ is the Kronecker function.}
\end{lemma}
\begin{proof}
First, for $j=1,..., m$, we know
\begin{equation}\label{gradj}
    \nabla (|x-\ov{x}_j|^{\alpha_j' p'}) = \alpha_j' p'  |x-\ov{x}_j|^{\alpha_j' p'-2}(x-\ov{x}_j)\,,
\end{equation}
\begin{equation}\label{lapj}
    \Delta (|x-\ov{x}_j|^{\alpha_j' p'}) = \alpha_j' p'(\alpha_j' p' -2 +n)  |x-\ov{x}_j|^{\alpha_j' p'-2}\,.
\end{equation}
    \par We argue by induction on $m \in \N$.\\
    For $m=1$, the statement is true since \eqref{lapj} holds.\\
    Let consider $m\in\N$, and suppose the statement to be true for $m-1$. We have
\begin{align*}
\Delta(w'^{\,p'})=&\Delta\left(\prod_{j=1}^m|x-\ov{x}_j|^{\alpha_j' p'}\right) = \\ &\Delta\left(\prod_{j=1}^{m-1}|x-\ov{x}_j|^{\alpha_j' p'}\right)|x-\ov{x}_m|^{\alpha_m' p'} + 
\left(\prod_{j=1}^{m-1}|x-\ov{x}_j|^{\alpha_j' p'}\right)\Delta(|x-\ov{x}_m|^{\alpha_m' p'})\,+\\ &  2\, \nabla\left( \prod_{j=1}^{m-1}|x-\ov{x}_j|^{\alpha_j' p'} \right) \cdot \nabla(|x-\ov{x}_m|^{\alpha_m' p'}) =: A_1+A_2+A_3.
\end{align*}
By inductive hypothesis
    \begin{align*}
    A_1\leq& \left( \left(\sum_{i=1}^{m-1}\alpha_{i}' p' + n -2\right) \sum_{h=1}^{m-1}\alpha'_{h}p' \prod_{i=1}^{m-1} |x-\ov{x}_i|^{\alpha'_ip'-2\delta_{ih}}\right) |x-\ov{x}_m|^{\alpha_m'p'} =\\
    &\left(\sum_{i=1}^{m-1}\alpha_{i}' p' + n -2\right) \sum_{h=1}^{m-1}\alpha'_{h}p' \prod_{i=1}^{m} |x-\ov{x}_i|^{\alpha'_ip'-2\delta_{ih}}\,.
    \end{align*} 
Moreover,
    \begin{align*}
        A_2 =& \left(\prod_{j=1}^{m-1}|x-\ov{x}_j|^{\alpha_j' p'}\right)\left( \alpha_m' p'(\alpha_m' p' -2 +n)  |x-\ov{x}_m|^{\alpha_m' p'-2} \right)= \\
        & \alpha_m' p'(\alpha_m' p' -2 +n) \left(\prod_{j=1}^{m}|x-\ov{x}_j|^{\alpha_j' p'-2\delta_{jm}}\right),
    \end{align*} 
and
    $$\nabla\left( \prod_{j=1}^{m-1}|x-\ov{x}_j|^{\alpha_j' p'} \right) = \sum_{h=1}^{m-1}\alpha'_hp'|x-\ov{x}_h|^{\alpha'_hp'-2}(x-\ov{x}_h) \prod_{\substack{i=1 \\ i \neq h}}^{m-1}|x-\ov{x}_i|^{\alpha'_ip'}. $$
Recalling \eqref{gradj},
    $$A_3 = 2\left( \sum_{h=1}^{m-1}\alpha'_hp'|x-\ov{x}_h|^{\alpha'_hp'-2}(x-\ov{x}_h) \prod_{\substack{i=1 \\ i \neq h}}^{m-1}|x-\ov{x}_i|^{\alpha'_ip'} \right) \cdot (\alpha_m' p'  |x-\ov{x}_m|^{\alpha_m' p'-2}(x-\ov{x}_m)). $$
Since $(x-\ov{x}_h) \cdot (x-\ov{x}_m) \leq \frac{1}{2} |x-\ov{x}_h|^2 + \frac{1}{2} |x-\ov{x}_m|^2$, we get
    \begin{align*}
    A_3 \leq & \left( \sum_{h=1}^{m-1}\alpha'_hp'|x-\ov{x}_h|^{\alpha'_hp'} \prod_{\substack{i=1 \\ i \neq h}}^{m-1}|x-\ov{x}_i|^{\alpha'_ip'} \right)\alpha_m' p'  |x-\ov{x}_m|^{\alpha_m' p'-2}\,+\\
    &\left( \sum_{h=1}^{m-1}\alpha'_hp'|x-\ov{x}_h|^{\alpha'_hp'-2} \prod_{\substack{i=1 \\ i \neq h}}^{m-1}|x-\ov{x}_i|^{\alpha'_ip'} \right)\alpha_m' p'  |x-\ov{x}_m|^{\alpha_m' p'}\,=\\
    & \alpha_m' p'\left( \sum_{h=1}^{m-1}\alpha'_hp'\prod_{i=1}^{m}|x-\ov{x}_i|^{\alpha'_ip' -2 \delta_{im}} \right)  + \alpha_m' p'\left( \sum_{h=1}^{m-1}\alpha'_hp' \prod_{i=1}^{m}|x-\ov{x}_i|^{\alpha'_ip'-2\delta_{ih}} \right)\,.
        \end{align*}
    Hence, finally we have 
    \begin{align*}
        \Delta(w'^{\,p'}) \leq & \left(\sum_{i=1}^{m-1}\alpha_{i}' p' + n -2\right) \sum_{h=1}^{m-1}\alpha'_{h}p' \prod_{i=1}^{m} |x-\ov{x}_i|^{\alpha'_ip'-2\delta_{ih}} \, +\\
        &\alpha_m' p'(\alpha_m' p' -2 +n) \left(\prod_{j=1}^{m}|x-\ov{x}_j|^{\alpha_j' p'-2\delta_{jm}}\right)\,+\\
        &\alpha_m' p'\left( \sum_{h=1}^{m-1}\alpha'_hp'\prod_{i=1}^{m}|x-\ov{x}_i|^{\alpha'_ip' -2 \delta_{im}} \right)  + \alpha_m' p'\left( \sum_{h=1}^{m-1}\alpha'_hp' \prod_{i=1}^{m}|x-\ov{x}_i|^{\alpha'_ip'-2\delta_{ih}} \right),
    \end{align*}
and, summing the first term with the last one and the second term with the third one, we have
    $$\Delta(w'^{\,p'}) \leq \left(\alpha' p' + n -2\right) \sum_{h=1}^{m-1}\alpha'_{h}p' \prod_{i=1}^{m} |x-\ov{x}_i|^{\alpha'_ip'-2\delta_{ih}} + (\alpha'p'-2+n)\alpha'_mp'\prod_{i=1}^{m}|x-\ov{x}_i|^{\alpha'_ip'-2\delta_{im}}, $$
that is 
$$\Delta(w'^{\,p'}) \leq (\alpha' p' + n -2) \sum_{h=1}^m (\alpha'_hp')
\prod_{i=1}^m |x-\ov{x}_i|^{\alpha'_i\,p'-2\delta_{hi}} .$$
Noting that $(\alpha' p' + n -2)\alpha'_hp'<0$ for all $h=1,\dots,m$, the proof is completed.
\end{proof}
\subsection{An interpolation inequality}
In this section, we provide an interpolation-type inequality for our specific weighted space. To this end, we recall the following:
\begin{lemma}\label{IL}
\sl
Let \( g \in L^\theta \cap L^s \), \( 1 \leq \theta < s \leq \infty \). Then there exists a constant \( c \), independent of \( g \), such that
\begin{equation}\label{II}
    \|g\|_{L^{r,1}} \leq c \|g\|_{s}^{\gamma} \|g\|_{\theta}^{1 - \gamma}, 
\end{equation}
provided that, for \( \gamma \in (0,1) \),
\begin{equation}\label{22}
    \frac{1}{r} = \frac{\gamma}{s} + \frac{1 - \gamma}{\theta}.
\end{equation}
\end{lemma}
\begin{proof}
See Lemma 3.2 in \cite{M_lorentz}.
\end{proof}

\begin{lemma} \label{interpol}
\sl
Let $g \in L^{\infty}(\R^n) \cap L^p_{w}(\R^n)$.There exists a positive constant $c$, independent of $g$, such that
\begin{equation}\label{stima_Lq}
\|g\|_q \leq c \|g\|_{w,p}^{\frac{n}{q}}\|g\|_{\infty}^{n(\frac{1}{n}-\frac{1}{q})}\,,
\end{equation}
for all $q>n$.
\end{lemma}
\begin{proof}
Under our hypotheses, $g\in L^q(B_R(O))$ for all $R>0$.\par
Taking in account Lemma \ref{IL}, choosing in \rf{II}
\begin{gather*}
r=\frac{n}{n-1}\,, \quad \theta=\frac{q}{q-1}\,, \quad s=\infty\,,
\end{gather*}
from \eqref{22}, we get 
$$\gamma=\frac{q-n}{n(q-1)}\,, \quad 1-\gamma=\frac{q(n-1)}{n(q-1)}\,,$$
where, in order to have $\gamma\in (0,1)$, we need $q>n$. Hence, substituting, we have
\begin{equation}\label{II2}
\dm|g|^{q-1}\dm_{L^{\frac{n}{n-1},1}(B_R)}\leq \dm |g|^{q-1}\dm_{L^{\frac{q}{q-1}}(B_R)}^{1-\gamma}\dm |g|^{q-1}\dm_{L^\infty(B_R)}^{\gamma}\leq \dm g\dm_{L^{q}(B_R)}^{(1-\gamma)(q-1)}\dm g\dm_{\infty}^{\gamma(q-1)}
\end{equation}
Taking this premise into account, we proceed to compute $\dm g\dm^q_{L^q(B_R)}$. Applying the H\"older's inequality for Lorentz spaces, we have
\begin{equation}\label{IF1}
\begin{array}{ll}
     \dm g\dm^q_{L^q(B_R)}=&\int_{B_R}|g|^q\,dx=\int_{B_R}|g|^{q-1}\,w^{-1}\,|g|w\leq c \|w^{-1}\|_{L^{\frac{n}{\alpha},\infty}(B_R)}\dm|g|^{q-1}|g|w\dm_{L^{\frac{n}{n-\alpha},1}(B_R)}\leq \\
     & c \, \|w^{-1}\|_{L^{\frac{n}{\alpha},\infty}(\R^n)}\dm|g|^{q-1}|g|w\dm_{L^{\frac{n}{n-\alpha},1}(B_R)}\leq c\, \dm|g|^{q-1}|g|w\dm_{L^{\frac{n}{n-\alpha},1}(B_R)} \,.
\end{array}
\end{equation}
where the norm $ \dm w^{-1}\dm_{L^{\frac{n}{\alpha},\infty}(\R^n)}$ is absorbed  into the constant $c$, since it is  finite and fixed.\\
Concerning the latter term, by the multiplication Theorem with
$$\frac{n-\alpha}{n}= \frac{1}{p}+\frac{n-1}{n}\,,\quad 1=\frac{1}{p}+\frac{1}{p\,'}\,,   $$
recalling 
$$\dm |g|^{q-1}\dm_{L^{\frac{n}{n-1},p' }(B_R)}\leq \dm|g|^{q-1}\dm_{L^{ \scriptstyle \frac{n}{n-1},1}(B_R)}\,, $$
we get
\begin{equation*}
\dm|g|^{q-1}|g|w\dm_{L^{\frac{n}{n-\alpha},1}(B_R)}\leq  c\,  \dm |g|w\dm_{L^{p,p}(B_R)}\dm |g|^{q-1}\dm_{L^{\frac{n}{n-1},p'}(B_R)}	
 \leq c \,\dm |g|w\dm_{L^{p,p}(B_R)}\dm |g|^{q-1}\dm_{L^{ \frac{n}{n-1},1}(B_R)}
\end{equation*}
Therefore, using the last relation in \eqref{IF1} , we conclude that
\begin{align*}
  \dm g\dm^q_{L^q(B_R)}\leq c\, \dm g\dm_{w,p}\,\dm|g|^{q-1}\dm_{L^{\frac{n}{n-1},1}(B_R)}\,,
\end{align*}
with $c$ independent of $R$. Consequently, 
\begin{equation*}
\dm g\dm^q_{L^q(B_R)}\leq c \dm g\dm_{w,p}\,\dm|g|^{q-1}\dm_{L^{\frac{n}{n-1},1}(B_R)}\,.
\end{equation*}   
Recalling \rf{II2}, we have
$$\dm g\dm^q_{L^q(B_R)}\leq c \dm g\dm_{w,p}\,\dm g\dm_{L^{q}(B_R)}^{(1-\gamma)(q-1)}\dm g\dm_{\infty}^{\gamma(q-1)}.$$
Dividing both sides by $\dm g\dm_{L^{q}(B_R)}^{(1-\gamma)(q-1)}$, and simplifying the exponents, we arrive at
\begin{equation*}
\dm g\dm_{L^q(B_R)}\leq c \dm g\dm_{w, p}^{\frac{q'}{q}(\frac{1}{q'-1+\gamma})}\, \dm g\dm_{\infty}^{\frac{\gamma}{q'-1+\gamma}}\,,
\end{equation*}
where
\begin{equation*}
\frac{q'}{q}\Big(\frac{1}{q'-1+\gamma}\Big)=\frac{n}{q }\,,\quad \frac{\gamma}{q'-1+\gamma}= n\Big(\frac{1}{n}-\frac{1}{q}\Big).
\end{equation*}
Applying Beppo Levi's Theorem, taking $R\to \infty$, we arrive at
\begin{equation*}
\dm g\dm_q \leq c \dm g\dm_{w,p}^{\frac{n}{q}}\,\dm g\dm_{\infty}^{n(\frac{1}{n}-\frac{1}{q})}\,,
\end{equation*}
which concludes the proof.
\end{proof}
\section{The Stokes Cauchy problem}\label{cauchy oroblem}
In this section, we consider  the Stokes Cauchy problem
\begin{equation}\label{Cproblem}
\begin{array}{ll}
u_t + \nabla \pi= \Delta u \,,\;&\text{in}\,\, (0,\infty) \times \mathbb{R}^n\,, \\
\nabla \cdot u=0\,, &\text{in}\,\, (0,\infty) \times \mathbb{R}^n, \\
u(0,x)=u_0(x)\,, 
\end{array}
\end{equation}
under the assumptions that $u_0 \in L^p_{w}(\mathbb{R}^n)$ and $\nabla \cdot u_0=0$ in the weak sense.\par
In \cite{MP}, the authors consider this problem with initial data in $L^p_{\alpha}(\R^n)=\{u\, :\, u|x|^{\alpha} \in L^p(\R^n)\}$ and, among other things, they establish results on the existence and uniqueness of a regular solution. In the present paper, we extend the results considering the more general weight function \eqref{wf}. To this end, we recall the following results.
\begin{lemma}\label{SHE}{\sl Let $\varphi_0\in \mathscr C_0(\R^n)$. Then we get a unique solution to problem \rf{Cproblem} such that
\be\label{RHE}\varphi(t, x)\in C^k(0,T;{\underset{q>1}\cap} J^q(\R^n))\cap({\underset{q>1}\cap} L^q(0,T;W^{2,q}(\R^n))\,,\mbox{ for }k\in\N_0\mbox{ and for all }T>0\,,\ee
and the representation formula holds:\be\label{RHEO} 
\varphi(t,x)=\int_{\R^n} H(t,x-y)\varphi_0(y)\,dy\,.
\ee
In particular, for all $r\geq q>1$, there exists a constant $c$ such that \be\label{SP}\dm \varphi(t)\dm_r+(t-s)^\frac12\dm\n \varphi(t)\dm_r\leq c(t-s)^{-\frac n2\left(\frac1q-\frac1r\right)}\dm \varphi(s)\dm_q\,,\mbox{ for all }t>s\geq0\,.\ee 
}\end{lemma}
\bp 
The existence, uniqueness and properties \rf{RHE}-\rf{RHEO} are classical results, and we refer to monograph \cite{LSU}. Concerning \rf{SP}, it also is classic, in any case it is an easy consequence of Young's Theorem applied to \rf{RHEO}.  \ep
The solutions introduced in the previous Lemma will hereafter be referred to as regular solutions.
\begin{lemma}\label{Linf_GM}
\sl{
Let $g\in L^p_w(\R^n)$. Then for the convolution product $H \ast g(t,x)$ the following estimate holds
\begin{equation}\label{GMELinf}
    t^{\frac{1}{2}}\dm H \ast g(t,x)\dm_{\infty} \leq c\, \dm g\dm_{w,p}\,,\qquad \mbox{for all }t>0\,.
\end{equation}
}
\end{lemma}
\begin{proof}
The proof follows the same line of argument as in Lemma 6 of \cite{GM}, noting that $w^{-1}\in L^{\frac{n}{\alpha},\infty}(\R^n)$
\end{proof}
Adopting the approach employed in \cite{MP}, we prove the following 
\begin{lemma}\label{SDS}
\sl{Let $\varphi_0\in \mathscr{C}_0(\mathbb{R}^n)$. Then, there exists a unique regular solution $\vp(t,x)$ to the problem \rf{Cproblem} with initial datum $\varphi_0$ such that
\begin{equation}\label{stima_per_convergenza}
\dm \vp(t)\dm_{L^{p'}_{w'}(\R^n)} \leq \dm \vp_0\dm_{L^{p'}_{w'}(\R^n)}\,,
\end{equation}
with $p'\in \Big(1,\frac{n}{n-1}\Big),$ and for all $t> 0$.
Moreover, the following property holds
\be\label{FTS}
\lim_{t\to0^+}\dm \varphi(t)-\varphi_0\dm _{L^{p'}_{w'}(\R^n)}=0\,.
\ee
}
\end{lemma}
\bp
The existence and uniqueness of a regular solution $(\varphi(t,x),\pi_{\varphi}(t,x))$, with $\pi_{\varphi}(t,x)\equiv c(t)$, are ensured by previous Lemma. So, we limit ourselves to proving property \eqref{stima_per_convergenza}.\par
Following the proof of Theorem 4 in \cite{MP}, we consider a smooth cutoff function  $h_R \in [0,1]$ with
\begin{equation*}
h_R =
\begin{cases}
1 \quad \text{if}\,\, |x|\leq R \,,\\
0 \quad \text{if}\,\, |x|\geq 2R\,,
\end{cases}
\end{equation*}
and $\nabla h_R=O(\frac{1}{R})$. Multiplying   the first equation in \eqref{Cproblem} by $h_R\,\varphi(\sigma^2+|\varphi|^2)^{\frac{p'-2}{2}}w^{\,\alpha'  p'}$, with $\sigma>0$, by means of an integration by parts on $\R^n$, we obtain
\be\label{int_parti_SE}\hskip-0.2cm\ba{ll}\displ
&\displ\sfrac{d}{dt}\dm h_R^\frac1{p'\hskip-0.1cm}(|\varphi(t)|^2+\sigma^2)^{\frac{1}{2}}\dm _{w' , p'}^{p'}+ p'\!\intl{\R^n}\!h_R\,|\nabla \varphi(t,x)|^2(\sigma^2+|\varphi(t, x)|^2)^{\frac{p'-2}{2}}w^{\,\alpha' p'}\,dx   \\
&\displ\hskip3.1cm+ p'(p'-2)\!\intl{\R^n}h_R\Big(\nabla \varphi(t,x)\varphi(t,x) \Big)^2(\sigma^2+|\varphi(t, x)|^2)^{\frac{p'-4}{2}}w^{\,\alpha' p'} dx \\
&\displ\hskip3.6cm=-I_R(t,\sigma)+\intl{\R^n}\!h_R(\sigma^2\hskip-0.1cm+|\varphi(t, x)|^2)^{\frac{p'-2}{2}}\Delta w^{\,\alpha'  p'}dx ,
\ea\ee where we set
$$I_R(t,\sigma):= \intl{\R^n}w^{\,\alpha'p'}(\sigma^2+|\varphi|^2)^\frac{p'}2 \Delta h_R dx+2 \intl{\R^n} \n h_R\cdot\n w^{\,\alpha'p'}(\sigma^2+|\varphi|^2)^{\frac{p'}{2}}dx\,.$$
Using \eqref{lapneg} in \rf{int_parti_SE}, we have
\begin{equation*}
\hskip-0.2cm\ba{ll}\displ
&\displ\sfrac{d}{dt}\dm h_R^\frac1{p'\hskip-0.1cm}(|\varphi(t)|^2+\sigma^2)^{\frac{1}{2}}\dm _{w' , p'}^{p'}+ p'\!\intl{\R^n}\!h_R\,|\nabla \varphi(t,x)|^2(\sigma^2+|\varphi(t, x)|^2)^{\frac{p'-2}{2}}w^{\alpha' p'}\,dx   \\
&\displ\hskip3.1cm+ p'(p'-2)\!\intl{\R^n}h_R\Big(\nabla \varphi(t,x)\varphi(t,x) \Big)^2(\sigma^2+|\varphi(t, x)|^2)^{\frac{p'-4}{2}}w^{\alpha' p'} dx \\
&\displ\hskip0.5cm\leq -I_R(t,\sigma)+[\alpha' p'+n-2]\sum_{h=1}^{m}\alpha'_{h}p'  \!\intl{\R^n}\!h_R(\sigma^2\hskip-0.1cm+|\varphi(t, x)|^2)^{\frac{p'-2}{2}}\prod_{j=1}^m |x-x_j|^{\alpha'_{j}p'-2\delta_{hj}}dx .
\ea\end{equation*}
Integrating both sides of \rf{int_parti_SE} on $(0,t)$, recalling \eqref{RHE} for $q=p'$ and letting $\sigma \to 0$ first and subsequently $R\to \infty$, by Lebesgue's dominated convergence Theorem, we have:
\begin{align*}
&\dm \varphi(t)\dm _{w' , p'}^{p'}+ p' \int_0^t\intl{\R^n}|\nabla \varphi(t,x)|^2|\varphi(t,x)|^{p'-2}w^{\alpha' p'}\,dx\,d\tau+ \\ 
& \hskip3.1cm+ p'(p'-2)\int_0^t\intl{\R^n}\Big(\nabla \varphi(t,x)\varphi(t,x) \Big)^2|\varphi(t,x)|^{p'-4}w^{\alpha' p'}\, dx\\
&\hskip3cm\leq[\alpha' p'+n-2]\sum_{h=1}^{m}\alpha'_{h}p' \int_0^t \intl{\R^n}|\varphi|^{p'}\prod_{j=1}^m |x-x_j|^{\alpha'_{j}p'-2\delta_{hj}}\,dx + \dm \varphi_0\dm^{p'}_{w',p'}.
\end{align*}
By \eqref{lapneg} again, the first term on the right-hand side  in the last equation is a negative quantity. So, for all $t>0$, we can deduce that
\begin{equation*}
\dm \varphi(t)\dm_{w',p'}\leq \dm\varphi_0\dm _{w',p'}
\end{equation*}
for $p' \in (1, \frac{n}{n-1})$ and for all $t>0$. \par
Finally, testing the equations of the problem \rf{Cproblem} in $\varphi(t, x)$, one easily achieves the weak continuity, that is $(\varphi(t),\phi)\in C((0,T))$, for all $\phi\in L^p_w(\R^n)$. Hence, employing  estimate \rf{stima_per_convergenza}, one deduces the property \rf{FTS}.
\ep
We are now in a position to prove the result corresponding to Theorem 4 in \cite{MP} in $L^p_w(\R^n)$-setting, where $w$ is given by \eqref{wf}.
\begin{theorem}\label{SPWS}
\sl{
Let consider $n\geq 3$, $p>n$. For all $u_0 \in J_w^p(\R^n)$ there exists a unique  smooth solution $u(t,x)$ to the Cauchy problem \eqref{Cproblem} such that
\begin{equation}\label{DCWS}
\dm u(t)\dm _{L^{p}_w(\R^n)}\leq \dm u(s)\dm _{L^{p}_w(\R^n)}\,,\mbox{ for all }t>s \geq 0\,.
\end{equation}
Moreover, there exists a constant $c$ independent of $u_0$ and of $s\geq0$, such that for all $l,k \in \N_0$ and for $q \in (n,\infty)$ the following estimate holds
\begin{equation}\label{gradienti_eq_stokes_beta_q}
\dm D_t^l\nabla^k u(t)\dm _{q}\leq c(t-s)^{-\frac{k}{2}-l-\frac{n}{2}({\frac{1}{n}}-\frac{1}{q})}\dm u(s)\dm _{L^{p}_w(\R^n)}\,, \text{ for all } t>0\,.
\end{equation} 
Furthermore, the following limit property holds
\begin{equation}\label{convergenza_parte_lineare}
\lim_{t \to 0^+} \dm u(t) - u_0\dm _{L^{p}_w(\R^n)} = 0\,.
\end{equation} 
  Finally,  for all $T>0$ and $q>n$, we get \be\label{CRHE}u\in C([0,T);J^p_w(\R^n))\cap C((\vep,T); L^q(\R^n)),\mbox{ for all }\vep>0\,.\ee
  }
\end{theorem}
\begin{proof}
The proof is exactly as in \cite{MP}, for the sake of completeness, we reproduce it below. We consider
\be\label{FRSU}u(t,x)=\int_{\R^n} H(t,x-y)u_0(y)\,dy\,,
\ee
the solution to problem \eqref{Cproblem} with initial datum $u_0\in \mathscr C_0(\R^n)$. By the H\"older's inequality and by estimate \eqref{stima_per_convergenza}, we get
\be\label{DALQ}\ba{ll}
(u(t),\varphi_0)\hskip-0.2cm&=\!\displ \int_{\R^n}\!\!\Big[\int_{\R^n}\!\!\!H(t,x-y)u_0(y)\,dy\Big]\varphi_0(x)dx=\!\int_{\R^n}\!\!\Big[\int_{\R^n}\!\!\!H(t,x-y)\varphi_0(x)\,dx\Big]u_0(y)dy\VS
&=\displ \int_{\R^n}\varphi(t,y)u_0(y)dy\leq \dm \varphi(t)\dm _{w',p'}\dm u_0\dm _{w,p}\leq \dm \varphi_0\dm _{w',p'}\dm u_0\dm _{w,p}\,,
\ea\ee
from which we obtain:
\begin{equation}\label{stima_alfa_p}
\dm u(t)\dm _{w,p}\leq \dm u_0\dm _{w,p}\,,
\end{equation}
for $p>n$, and for all $t>0$. Testing \rf{Cproblem} in $u(t, x)$, one deduces the weak continuity, and employing \rf{stima_alfa_p}, we get 
    $$\lim_{t\to0}\dm u(t)-u_0\dm_{w, p}=0\,.$$ Differentiating  \rf{FRSU}, and then employing the duality argument used in \rf{DALQ} for $u_t$, for all $t>0$, we also get
\be\label{RGE-I}\dm u_t(t)\dm_{w,p}\leq \dm \Delta u_0\dm_{w,p}\,.\ee
Hence, as a consequence of estimates \rf{stima_alfa_p}, \rf{RGE-I}, and of regularity of $u(t, x)$, we get \be\label{RGE-III}u\in C([0,T);L^{p}_{w}(\R^n))\,,\mbox{ for all }T>0\,.\ee
{{Now, we need to estimate $D_t^l \nabla_x^k u(t,x)$. Firstly, we observe that by virtue of Lemma \ref{interpol}, estimates \eqref{GMELinf} and \eqref{stima_alfa_p} , for all $q \in (n,\infty)$ and $t>0$, we have:
\be\label{SLQ}\dm u(t)\dm_{q}\leq c \dm u(t) \dm^{1-\gamma}_{w,p} \dm u(t)\dm^{\gamma}_{\infty}\leq c t^{-\frac{n}{2}(\frac{1}{n}-\frac{1}{q})}\dm u_0\dm_{w,p}\ee
Now, recalling the well known estimate 
\begin{equation} \label{usare_per_conv_2}
\dm D_t^l\nabla^k u(t-\sigma)\dm _{q}\leq c (t-\sigma)^{-\frac{k}{2}-l} \dm u(\sigma)\dm_q\,,  
\end{equation}
for all $l,k \in \N_0$ and $t>\sigma\geq0$,}} choosing $\sigma=\frac t2$ and then employing \rf{SLQ}, we arrive at 
\be\label{LQ-LW}\dm D_t^l\nabla^k u(t)\dm _{q}\leq ct^{-\frac{k}{2}-l-\frac{n}{2}(\frac{1}{n}-\frac{1}{q})}\dm u_0\dm _{w, p}\,.\ee 
At this point, we consider $u_0\in J_{w}^p(\R^n)$. There exists a sequence $\{u_0^m\}\subset \mathscr{C}_0(\R^n)$ such that $u_0^m \to u_0$ in $J_{w}^p(\R^n)$ and we denote by $\{u^m(t,x)\}$ the sequence of smooth solutions to the Cauchy problem \rf{Cproblem}.
By virtue of the linearity of the problem  and \eqref{LQ-LW},  we also  get:
\begin{gather*}
\dm u^m(t)-u^{\nu}(t)\dm _{w, p}\leq \dm u_0^m-u_0^{\nu}\dm _{w,p}\\
\dm D_t^l\nabla^k u^m(t)-D_t^l\nabla^k u^{\nu}(t)\dm _{q}\leq ct^{-\frac{k}{2}-l-\frac{n}{2}({\frac{1}{n}}-\frac{1}{q})}\dm u^m_0-u^{\nu}_0\dm _{w, p},
\end{gather*}
which ensure the existence of a limit $u(t,x)$ uniformly in $t>0$ and $u$ enjoys \rf{CRHE}.\\
Finally, by \rf{CRHE}, together with \eqref{stima_alfa_p} and \eqref{LQ-LW}, one easily proves \eqref{DCWS} and \eqref{gradienti_eq_stokes_beta_q} for all $t>s\geq0$. 
\par[{\it Uniqueness}] As it is usual for the $L^p$-theory, also in our special $L^p$-spaces,  we use a duality argument to state   the uniqueness. Let  $(u(t,x), \pi_u(t,x))$ and $(v (t,x), \pi_v(t,x))$ be two solutions to the Cauchy Stokes problem \eqref{Cproblem} with initial datum $u_0$. We set $w(t,x): = u (t,x)-v (t,x)$ and $\pi_w(t, x):=\pi_u(t, x)-\pi_v(t, x)$. Then, $w$ is a solution to following problem:
\begin{equation} \label{eq_w1}
\begin{array}{ll}
w_t - \Delta w = -\n\pi_w\qquad &\text{in} \quad (0,\infty)\times \R^n\,, \\
\n \cdot w(t,x)=0 \qquad &\text{in} \quad (0,\infty)\times \R^n \,,\\
w(0,x)=0\qquad &\text{on}\quad  \{0\}\times \R^n.
\end{array}
\end{equation}
Let $\varphi_0 \in \mathscr{C}_0(\R^n)$ 
and let $(\varphi(t,x),c(t))$ be a solution to the Stokes Cauchy problem \eqref{Cproblem} corresponding to $\varphi_0$ enjoying \rf{RHE}, \rf{stima_per_convergenza} and  \rf{FTS}. We define:
\begin{equation*}
\widehat{\varphi}(\tau,x) = \varphi(t-\tau,x)\,, \qquad \text{for} \,(\tau,x) \in (0,t)\times \R^n\,.
\end{equation*}
As it is known $\widehat{\varphi}$ is a solution backward in time on $(0,t)\times\R^n$ with $
\widehat{\varphi}(t,x) = {\varphi}_0(x) $ on $  \{t\}\times \R^n
$.
Multiplying the first equation in  \eqref{eq_w1} for $\widehat{\varphi}$ and integrating by parts on $[s,t] \times \R^n$, we obtain:
\begin{equation*}
(w(t),\varphi_0) = (w(s),\varphi(t-s))\,,
\end{equation*}
from which 
\begin{equation*}
|(w(t),\varphi_0)| = |(w(s),\varphi(t-s))| \leq \dm w(s)\dm _{w,p}\dm \varphi(t-s)\dm _{w',p'}\,.
\end{equation*}
Using \eqref{stima_per_convergenza},
we get 
\begin{equation*}
|(w(t),\varphi_0)| \leq \dm w(s)\dm _{w,p} \dm \varphi_0\dm _{w',p'}\,.
\end{equation*}
The second term tends to $0$ as $s \to 0$, so $$(w(t),\varphi_0) = 0\,.$$
Due to the arbitrariness of $\varphi_0$ and $t$, the last relation implies that $w = 0 $ on $\{t\} \times \R^n$ for all  $t>0$.
This completes the proof. \end{proof}
\section{The IBVP of the Stokes system in the half-space}\label{IBVP problem}
We consider the problem \eqref{problem}, which has been extensively studied in the framework of classical Lebesgue spaces. In particular, we recall the results due to Solonnikov in \cite{sol}, \cite{soln}. In the aforementioned works, the author deduces the Green function for this problem as $\mathcal{G} = (G_{ij})_{i,j=1,...,n}$, with
$$G_{ij}:= \delta_{ij}(H(x-y,t)-H(x-y^*,t)) + 4(1-\delta_{jn})\dfrac{\partial
}{\partial x_j} \int_0^{x_n}\int_{\R^{n-1}} \dfrac{\partial}{\partial x_i} \mathcal{E}(x-y) H(y-z^*,t)\,dz,$$
and $\mathcal{P}=(P_{j})_{j=1,\dots,n}$ defined by the formula
\begin{align*}
  P_j(x,y,t)&=4\nu (1-\delta_{jn}) \frac{\partial}{\partial x_j}\Big[\int_{\R^{n-1}}\frac{\partial}{\partial x_n}\mathcal{E}(x'-z',x_n)\, H(z'-y',y_n,t)\, dz'\\[5pt]
  &+ \int_{\R^{n-1}}\frac{\partial}{\partial y_n}\mathcal{E}(x'-z',x_n)({x'-z',x_n})\,H(z'-y',y_n,t)\, dz'\Big]\,,
\end{align*}
where $y^*=(y_1,\dots, -y_n)$, $y'=(y_1,\dots,y_{n-1})$, and similarly for $z^*$, $z'$ and $x'$. Moreover, he has shown that the unique solution $((v(t,x),\pi(t,x))$, corresponding to an initial datum $v_0 \in \mathscr{C}_0(\R^n_+)$, can be expressed by the following representation formulae 
\begin{gather}
        v(t,x)=\int_{\R^n_+}\mathcal{G}(x,y,t)v_0(y) \, dy , \label{IRv}\\  
    \pi(t,x)=\int_{\R^n_+}\mathcal{P}(x,y,t)v_0(y) \, dy.
\end{gather}

\subsection{Proof of the main Theorem}
We first provide the following lemma, which is analogous to Lemma \ref{SDS} in the half-space case.
\begin{lemma}
\sl{
Let $\varphi_0\in \mathscr{C}_0(\mathbb{R}^n_+)$. Then, there exists a unique regular solution $\vp(t,x)$ to the problem \rf{problem} with initial datum $\varphi_0$ such that
\begin{equation}\label{DPRSHS}
\dm \vp(t)\dm_{L^{p'}_{w'}(\R^n_+)} \leq c\dm \vp_0\dm_{L^{p'}_{w'}(\R^n_+)}\,,
\end{equation}
with  $p' \in (1,\frac{n}{n-1})$, and for all $t\geq0$.
Moreover, the following property holds
\be\label{FTSHP}
\lim_{t\to0^+}\dm \varphi(t)-\varphi_0\dm _{L^{p'}_{w'}(\R^n_+)}=0\,.
\ee
}
\end{lemma}
\bp
Firstly, we consider $\varphi_0 \in \mathscr{C}_0(\R^n_+)$. Let $(\varphi(t,x),\pi_{\varphi}(t,x))$ be the smooth solution to problem \eqref{problem} corresponding to $\varphi_0$. As we recalled above, $\varphi(t,x)$ admits the representation \eqref{IRv} that can be rewritten as

\begin{align}\label{IRSS}
\varphi(t,x) =\; & \int_{\R^n} \left[ H(t,x-y) - H(t,x-y^*) \right] \vp_0(y)\,dy \notag \\
& + 4 \sum\displaylimits_{i=1}^{n-1} \frac{\partial}{\partial x_i} 
\int_0^{x_n} \int_{\R^{n-1}} \nabla \mathcal{E}(x - y)
\left( \int_{\R^n_+} H(t, y - z^*) \vp_0(z)\,dz \right) dy.
\end{align}

Now, we consider the following extensions for the initial datum $\vp_0$:
\begin{enumerate}
\item[i.] $\vp_0^*$, the odd extension to the whole space $\R^n$;
\item[ii.] $\overline{\vp}_0$, the extension defined by
\begin{equation*}
\overline{\vp}_0(x)= 
\begin{cases}
0 \quad &\mbox{if }x_n>0 \\
\vp_0(x^*)\quad &\mbox{if } x_n<0\,.
\end{cases}
\end{equation*}
\end{enumerate}
We can therefore rewrite \eqref{IRSS} as
\begin{align*}
    \vp(t,x)=&\int_{\R^n}H(t,x-y)\vp_0^*(y)\,dy\,+ \\
&4 \sum_{i=1}^{n-1}\frac{\partial}{\partial x_i}\int_0^{x_n}\int_{\R^{n-1}} \nabla \mathcal{E}(x-y)\Big(\int_{\R^n} H(t,y-z)\overline{\vp}_0(z)\,dz\Big)\,dy \,.
\end{align*}
To read it more clearly, we set
$$ \vp^*(t,x) := \int_{\R^n}H(t,x-y)\vp_0^*(y)\,dy, $$ 
$$\overline{\vp}(t,y):=\int_{\R^n} H(t,y-z)\overline{\vp}_0(z)\,dz\,, $$
and
$$S\,\overline{\vp}(t,x):=4 \sum_{i=1}^{n-1}\frac{\partial}{\partial x_i}\int_0^{x_n}\int_{\R^{n-1}} \nabla \mathcal{E}(x-y)\overline{\vp}(t,y)\,dy\,.$$
We explicitly note that $\vp^*(t,x)$ and $\overline{\vp}(t,y)$ are the solutions of the Stokes Cauchy problem in the $\R^n$ correspondents to initial data $\vp_0^*(x)$ and $\overline{\vp}_0(x)$, respectively, and $S$ is a Calder\'on-Zygmund singular operator. Therefore, we have 
\begin{equation}\label{dec}
\vp(t,x)=\vp^*(t,x) + S\,\overline{\vp}(t,x).
\end{equation}
So, by \eqref{stima_per_convergenza}, we get
\begin{equation} \label{ineq1}
 \dm \vp^*(t)\dm_{L^{p'}_{w'}(\R^n)}\leq \dm \vp^*_0\dm_{L^{p'}_{w'}(\R^n)}.  
\end{equation}
Moreover, by \eqref{WCZ} and \eqref{stima_per_convergenza}, we obtain
\begin{equation} \label{ineq2}
 \dm S\,\overline{\vp}(t)\dm_{L^{p'}_{w'}(\R^n)}\leq c\dm \overline{\vp}(t)\dm_{L^{p'}_{w'}(\R^n)} \leq  c\dm \overline{\vp}_0\dm_{L^{p'}_{w'}(\R^n)}.  
\end{equation}
By \eqref{dec}, using \eqref{ineq1}, \eqref{ineq2} and the continuity of extensions i. and ii.\footnote{see \cite{Chua}.}, we get
\begin{align} \label{SE1}
\dm \vp(t)\dm_{L^{p'}_{w'}(\R^n_+)}\leq \dm \vp^*_0\dm_{L^{p'}_{w'}(\R^n)} + c\,\dm \overline{\vp}_0\dm_{L^{p'}_{w'}(\R^n)}\leq c\, \dm \vp_0\dm_{L^{p'}_{w'}(\R^n_+)} , 
\end{align} 
Now, we prove \eqref{FTSHP}. At first, testing by $\vp$ we get weak continuity. Moreover, differentiating  \eqref{dec} with respect to time, we can observe that
$$\vp_t = \vp_t^* + (S\,\overline{\vp})_t = \vp_t^* + (S\,\overline{\vp}_t).$$
\par By virtue of the fact that $\vp_0 \in \mathscr{C}_0(\R^n_+)$, the following computations are allowed.\\
Using \eqref{stima_per_convergenza}, we get
$$\dm\vp_t^*(t)\dm_{L^{p'}_{w'}(\R^n)} \leq c\, \dm \Delta \vp^*_0\dm_{L^{p'}_{w'}(\R^n_+)}, $$
and, using \eqref{WCZ} together with  \eqref{stima_per_convergenza} again
$$\dm S\, \ov{\vp}_t (t) \dm_{L^{p'}_{w'}(\R^n)} \leq c \dm\ov{\vp}_t(t)\dm_{L^{p'}_{w'}(\R^n)} \leq c \dm \Delta \ov{\vp}_0\dm_{L^{p'}_{w'}(\R^n_+)}\,. $$
Therefore,
\begin{equation}\label{BRL}
    \dm \vp_\tau (\tau) \dm_{L^{p'}_{w'}(\R^n_+)} \, \leq c = c\,(\dm \Delta \vp_0 \dm_{L^{p'}_{w'}(\R^n_+) }).
\end{equation}
Hence, due to \eqref{BRL}, we deduce
$$\dm \vp(t)- \vp (s)\dm_{L^{p'}_{w'}(\R^n_+)} = \left|\left| \int_s^t  \vp_\tau (\tau)\, d \tau\right|\right|_{L^{p'}_{w'}(\R^n_+)} \leq c  \int_s^t \dm \vp_\tau (\tau) \dm_{L^{p'}_{w'}(\R^n_+)} \, d \tau \leq c \left| t-s\right|, $$
so $\vp \in C([0,T); L^{p'}_w(\R^n_+))$, and, in particular,
$$\lim_{t \to 0^+} \dm \vp(t)- \vp_0\dm_{L^{p'}_{w'}(\R^n_+)} =0\,,$$
which concludes the proof. 
\ep

\begin{proof}[Proof of the Theorem \ref{Theoremprincipale}]{[\textit{Existence}]}
Initially, as before, we consider $\varphi_0 \in \mathscr{C}_0(\R^n_+)$. Let $(\varphi(t,x),\pi_{\varphi}(t,x))$ be the smooth solution to the problem \eqref{problem} corresponding to $\varphi_0$. As in previous Lemma, we rewrite $\varphi(t,x)$ as
\begin{equation}\label{decTHM}
\vp(t,x)=\vp^*(t,x) + S\,\overline{\vp}(t,x).
\end{equation}
By the same arguments, involving estimate \eqref{DCWS}, we deduce that
\be \label{SE2}
 \dm \vp(t)\dm_{L^{p}_{w}(\R^n_+)}\leq c \dm \vp_0\dm_{L^{p}_{w}(\R^n_+)}\,,
\ee
for $p>n$ and $t>0$.
Now, we estimate the solution and its successive derivatives in the $L^q(\R^n_+)$-norm for all $t>0$, analogously to the approach adopted above. Indeed, let us  consider $q \in(n,\infty)$. Recalling the estimate \eqref{gradienti_eq_stokes_beta_q} with $l=k=0$, we get
\begin{equation}\label{inq1}
    \dm \vp^*(t)\dm_{L^q(\R^n)} \leq c\, t^{-\frac{n}{2}(\frac{1}{n}-\frac{1}{q})}\dm \vp_0^*\dm_{L^p_w(\R^n)}\,.
\end{equation}
Moreover, employing the classical Calder\'on-Zygmund Theorem we also have
\begin{equation}\label{inq2}
    \dm S\,\overline{\vp}(t)\dm_{L^q(\R^n)} 
    \leq c \dm \overline{\vp}(t)\dm_{L^q(\R^n)} 
    \leq c\, t^{-\frac{n}{2}(\frac{1}{n}-\frac{1}{q})}\dm \overline{\vp}_0\dm_{L^p_w(\R^n)}.
\end{equation}
Therefore, from \eqref{dec} and by the continuity of extensions i. and ii., there holds
\begin{equation}\label{inq}
    \dm \vp(t)\dm_{L^q(\R^n_+)} 
    \leq c\, t^{-\frac{n}{2}(\frac{1}{n}-\frac{1}{q})}\dm \vp_0\dm_{L^p_w(\R^n_+)}.
\end{equation}
For $L^\infty(\R^n_+)$ estimate, fix $t>0$. Setting $s=\frac{t}{2}$, we consider the following problem
\begin{equation}\label{Hproblem}
\begin{array}{ll}
\psi_\tau + \nabla \pi_{\psi}= \Delta \psi \,,\;&\text{in}\,\, (0,\infty) \times \mathbb{R}^n_{+}\,, \\
\nabla \cdot \psi=0\,, &\text{in}\,\, (0,\infty) \times \mathbb{R}^n_{+}, \\
\psi_{|x_n=0}\,=0 \,,\\
\psi(0,x)=\vp(s,x)\,. 
\end{array}
\end{equation}
From \eqref{inq}, we know that $\vp(s,x) \in L^q(\R^n_+)$, for  all $q\in (n,\infty)$, Concerning \eqref{Hproblem}, we apply the well known classical theory for the Stokes system. So, there exists a unique smooth solution $(\psi(\tau,x),\pi_{\psi}(\tau,x))$ to problem \eqref{Hproblem} such that the following estimate holds
$$
\dm \psi(\tau)\dm_{L^\infty(\R^n_+)}\leq  c\, \tau^{-\frac{n}{2q}}\dm\psi(0)\dm_{L^q(\R^n_+)},$$
for all $\tau \geq 0$. So, by uniqueness $\psi(\tau,x)= \vp(\tau + s,x)$, 
\begin{align*}
\dm \vp(\tau)\dm_{L^\infty(\R^n_+)}\leq  c (\tau-s)^{-\frac{n}{2q}}s^{-\frac{n}{2}(\frac{1}{n}-\frac{1}{q})}\dm\vp_0\dm_{L^p_w(\R^n_+)}\,.
\end{align*}
Choosing $\tau=t$ and recalling that $s=\frac{t}{2}$, we have
\be \label{LinftyE}
\dm \vp(t)\dm_{L^\infty(\R^n_+)}\leq  c\, t^{-\frac{1}{2}}\dm\vp_0\dm_{L^p_w(\R^n_+)}\,.
\ee
As $t$ is arbitrary,  estimate \eqref{LinftyE} holds for all $t>0$. This completes the $L^q$-estimate of $\vp(t,x)$ for all $q\in(n,\infty]$. \\
Finally, from classical theory, for all $l, k \in \mathbb{N}_0$, the following estimate also holds
\begin{equation*}
\dm D^l_t \nabla^k_x
\vp(t)\dm_{L^q(\R^n_+)}\leq c\, (t-\sigma)^{-\frac{k}{2}-l}\dm\vp(\sigma)\dm_{L^q(\R^n_+)}\,.
\end{equation*}
Thus, choosing $\sigma=\frac{t}{2}$, by \eqref{inq} and \eqref{LinftyE} we arrive at 
\begin{equation} \label{dersuc}
\dm D_t^l \nabla^k \vp(t)\dm_{L^q(\R^n_+)}\leq c\, t^{-\frac{n}{2}(\frac{1}{n}-\frac{1}{q})-\frac{k}{2}-l}\dm\vp_0\dm_{L^p_w(\R^n_+)}\,,
\end{equation}
for all $q\in (n,\infty]$, and $t>0$.\par
Due to \eqref{dersuc} for $(l,k) = (0,2)$ and $(l,k) = (1,0)$, we can estimate the pressure gradient using $\eqref{problem}_1$. So, we obtain
\begin{equation}\label{pres}
 \dm  \nabla \pi_\vp(t)\dm_{L^q(\R^n_+)}\leq c\, t^{-\frac{n}{2}(\frac{1}{n}-\frac{1}{q})-1}\dm\vp_0\dm_{L^p_w(\R^n_+)}\,,  
\end{equation}
for all $t> 0$, and $q>n$.
To prove that $\vp(t,x)$ attains the initial datum $\vp_0$ in $L^p_w(\R^n_+)$-norm, it suffices to argue as in the previous Lemma. \par
Now, we consider $v_0\in J_{w}^p(\R^n_+)$. There exists a sequence $\{v_0^m\}\subset \mathscr{C}_0(\R^n_+)$ such that $v_0^m \to v_0$ in $J_{w}^p(\R^n_+)$ and we denote by $\{v^m(t,x)\}$ the sequence of smooth solutions to the Cauchy problem \rf{Cproblem}.
By virtue of the linearity of the problem, \eqref{SE2},  and \eqref{dersuc},  we get:
\begin{gather*}
\dm v^m(t)-v^{\nu}(t)\dm _{w, p}\leq c \dm v_0^m-v_0^{\nu}\dm _{w,p}\\
\dm D_t^l\nabla^k v^m(t)-D_t^l\nabla^k v^{\nu}(t)\dm _{q}\leq ct^{-\frac{k}{2}-l-\frac{n}{2}({\frac{1}{n}}-\frac{1}{q})}\dm v^m_0-v^{\nu}_0\dm _{w, p}
\end{gather*}
which ensure the existence of a limit $v(t,x)$ uniformly in $t>0$ and $v$ enjoys \rf{regular_Hs}. Moreover, for the limit $v(t,x)$, \rf{SE2} and \rf{dersuc} trivially hold.

Let us resume the convergence property:
$$v_0^m \to v_0 \text{ in } J_{w}^p(\R^n_+),$$
$$\lim_{t \to 0^+} \dm v(t)- v_0\dm_{L^p_w(\R^n_+)} =0, $$
and 
$$ \dm v^{m_1}(t)- v^{m_2}(t)\dm_{L^p_w(\R^n_+)} \leq c \dm v^{m_1}_0- v^{m_2}_0\dm_{L^p_w(\R^n_+)}, \text{ for all } t>0, $$
so we deduce the uniform in time convergence of $v^m(t)$ to $v(t)$.
Let $\varepsilon>0$ be fixed, there exists $M(\varepsilon) \in \N $ such that
$\dm v_0^{M(\varepsilon)} - v_0\dm_{L^p_w(\R^n_+)} <\varepsilon$ and $\dm v^{M(\varepsilon)}(t) - v(t)\dm_{L^p_w(\R^n_+)} <\varepsilon$, for all $t>0$. Therefore, 
\begin{align*}
    \dm v(t) - v_0\dm_{L^p_w(\R^n_+)} \leq & \dm v(t) - v^{M(\varepsilon)}(t)\dm_{L^p_w(\R^n_+)} + \dm  v^{M(\varepsilon)}(t) - v_0^{M(\varepsilon)}\dm_{L^p_w(\R^n_+)} + \dm v_0^{M(\varepsilon)} - v_0\dm_{L^p_w(\R^n_+)} \leq \\
    & 2 \varepsilon \; + \; \dm  v^{M(\varepsilon)}(t) - v_0^{M(\varepsilon)}\dm_{L^p_w(\R^n_+)},
\end{align*}
so passing to the limit as $t\to 0^+$, by the arbitrariness of $\varepsilon$, we get \eqref{LP}.\\
Finally, by \rf{regular_Hs}, together with \eqref{SE2}, \eqref{dersuc}, and \eqref{pres} one easily proves \eqref{RP} and \eqref{pressure} for all $t>s\geq0$. 
\par[{\it Uniqueness}] As done before, we use a duality argument to state   the uniqueness. Let  $(u(t,x), \pi_u(t,x))$ and $(v (t,x), \pi_v(t,x))$ be two solutions to the IBVP of the Stokes equations \eqref{problem} with initial datum $v_0$. We set $w(t,x): = u (t,x)-v (t,x)$ and $\pi_w(t, x):=\pi_u(t, x)-\pi_v(t, x)$. So, $w$ is a solution to following problem:
\begin{equation}\label{eq_w}
\begin{array}{ll}
w_t + \nabla \pi_w= \Delta w \,,\;&\text{in}\,\, (0,\infty) \times \mathbb{R}^n_{+}\,, \\
\nabla \cdot w=0\,, &\text{in}\,\, (0,\infty) \times \mathbb{R}^n_{+}, \\
w_{|x_n=0}\,=0 \,,\\
w(0,x)=0\,. 
\end{array}
\end{equation}
Let $\varphi_0 \in \mathscr{C}_0(\R^n_+)$ 
and let $(\varphi(t,x),\pi_{\vp})$ a solution to the  problem \eqref{problem} corresponding to $\varphi_0$ enjoying \rf{RHE}, \rf{stima_per_convergenza} and  \rf{FTS}. We define:
\begin{equation*}
\widehat{\varphi}(\tau,x) = \varphi(t-\tau,x) \,,\qquad \text{for} \,(\tau,x) \in (0,t)\times \R^n_+\,.
\end{equation*}
As it is known, $\widehat{\varphi}$ is a solution backward in time on $(0,t)\times\R^n_+$ with $
\widehat{\varphi}(t,x) = {\varphi}_0(x) $ on $  \{t\}\times \R^n_+
$.
Multiplying the first equation in  \eqref{eq_w} for $\widehat{\varphi}$ and integrating by parts on $[s,t] \times \R^n_+$, we obtain:
\begin{equation*}
(w(t),\varphi_0) = (w(s),\varphi(t-s))\,,
\end{equation*}
from which 
\begin{equation*}
|(w(t),\varphi_0)| = |(w(s),\varphi(t-s))| \leq \dm w(s)\dm_{L^p_w(\R^n_+)}\dm \varphi(t-s)\dm_{L^{p'}_{w'}(\R^n_+)}\,.
\end{equation*}
Using \eqref{stima_per_convergenza},
we get 
\begin{equation*}
|(w(t),\varphi_0)| \leq \dm w(s)\dm_{L^p_w(\R^n_+)} \dm \varphi_0\dm_{L^{p'}_{w'}(\R^n_+)}\,.
\end{equation*}
The second term tends to $0$ as $s \to 0$, so $$(w(t),\varphi_0) = 0\, .$$
Due to the arbitrariness of $\varphi_0$ and $t$, the last relation implies that $w = 0 $ on $\{t\} \times \R^n_+$ for all  $t>0$.
This completes the proof.
\end{proof}

\section{Appendix} 
The arguments presented in this Appendix follow the approach developed by Solonnikov in \cite{sol} (see also \cite{Maremonti}).\par
We define the following spaces.
\begin{gather*}
J^q_{w}(\mathbb{R}_+^n)=\{v\in L^q_{w}(\mathbb{R}_+^n)\, |\, (v,\nabla h)=0 \,\, \text{for all}\, h \in W^{1,q'}_{loc}(\mathbb{R}_+^n,w')\text{,}\, \nabla h \in L^{q'}_{w'}(\R_+^n)\}\,,\\
G^q_{w}(\mathbb{R}_+^n)=\{u\in L^q_{w}(\mathbb{R}_+^n)\, |\, \exists \pi\, : \,\, u=\nabla \pi \text{,}\,\text{with} \, \pi \in W^{1,q}_{loc}(\mathbb{R}_+^n,w)\text{,}\, \nabla \pi \in L^{q}_{w}(\R_+^n)\}\,,
\end{gather*}
where $w$ and $w'$ are the weight function and its dual, as defined in \eqref{wf} and \eqref{dwf}, respectively.
We need of the following:
\begin{lemma} \label{lemma}{\sl
Let $q \in (1,+\infty)$. Then, there holds
\begin{equation}
J^q_{w}(\mathbb{R}_+^n)\cap G^q_{w}(\mathbb{R}_+^n)=\{0\}\,.
\end{equation}}
\end{lemma}
\begin{proof}
Assume that exists $f \in J^q_{w}(\mathbb{R}_+^n)\cap G^q_{w}(\mathbb{R}_+^n) $ such that $f\neq 0$. Then, by definition:
\begin{gather*}
(f,\nabla h)=0 \quad \text{for all}\, h \in W^{1,q'}_{loc}(\mathbb{R}_+^n,w)\text{,}\, \nabla h \in L^{q'}_{w}(\R_+^n)\,, \\
f=\nabla \pi \quad \text{with} \, \pi \in W^{1,q}_{loc}(\mathbb{R}_+^n,w )\text{,}\, \nabla \pi \in L^{q}_{w}(\R_+^n)\,.
\end{gather*}
Let $h_R \in (0,1)$ be a smooth cut-off function with:
\begin{equation*}
h_R =
\begin{cases}
1 \quad \text{if}\,\, |x|\leq R\,, \\
0 \quad \text{if}\,\, |x|\geq 2R\,,
\end{cases}
\end{equation*}
and $\nabla h_R=O(\frac{1}{R})$.\par
We denote by $\psi$ the solution for the problem:
\begin{equation} \label{NP}
\begin{array}{ll}
     & \Delta \psi= \nabla \cdot g \,,\quad \mbox{in }\R_+^n\,,\\
     & \frac{d\psi}{d\nu}=g\cdot \nu\,,  \quad \,\,\,\mbox{on }x_n=0\,,
\end{array}
\end{equation}
with $g\in C^{\infty}_0(\R_+^n)$ and $\nu$ the unit normal vector to the surface $x_n=0$. We know that
\begin{equation}
\psi = \int_{\R_+^n} \nabla_y[(\mathcal{E}(x-y)+\mathcal{E}(x-y*)]g(y)\,dy\,,
\end{equation}
and
\begin{equation}\label{86}
\nabla \psi= \int^{*}_{\R^n} \nabla \nabla [(\mathcal{E}(x-y)+\mathcal{E}(x-y*)]g(y)\,dy\,,
\end{equation}
where $\mathcal{E}$ is the fundamental solution of the Laplace equation.\par
We can employ Calder\'on-Zygmund Theorem to obtain the following estimate:
\begin{equation}
\dm \nabla \psi\dm _{L^{q'}(\R^n_+)}\leq C\dm g\dm _{L^{q'}(\R^n_+)}\,, \quad \text{for all}\, q' \in (1,+\infty).
\end{equation}
Moreover, by \eqref{WCZ}, we get
\begin{equation}\label{stein}
\dm \nabla \psi \dm _{L^{q'}_{w'}(\R^n_+)}\leq c \dm g \dm _{L^{q'}_{w'}(\R^n_+)}\,.
\end{equation}
Let $d=\displaystyle\max_{j=1,\dots,m} |\ov{x}_j|$, and consider $B^+_{2R}(O)$ the upper half-ball centered at the origin, i.e. $ B_{2R}(O)\cap \{x\,:\, x_n\geq0\}$  where $R> \displaystyle\max\{ diam(supp\,g), d\}$. Since $g\in C^{\infty}_0(\R^n_+)$, it follows that $\nu\cdot \nabla \psi=0$ on $x_n=0$. Therefore, we have
\begin{align*}
0&=(f,\nabla\psi)=(\nabla \pi,\nabla \psi)=(h_R \nabla \pi, \nabla \psi) + ( (1-h_R)\nabla \pi,\nabla \psi)\\
 &= -((\pi-\overline{\pi}_R)h_R,\nabla \cdot g) - ((\pi-\overline{\pi}_R)\nabla h_R,\nabla \psi)+( (1-h_R)\nabla \pi,\nabla \psi)\\
&= (\nabla\pi,g)-((\pi-\overline{\pi}_R)\nabla h_R,\nabla \psi)+  ((1-h_R)\nabla \pi,\nabla \psi)=: I_1 + I_2 + I_3\,,
\end{align*}
where $\overline{\pi}_R=\frac{1}{|B^+_{2R}|}\int_{B^+_{2R}}\pi(y) \, dy$.\\
Trivially, $I_3$ tends to zero as $R\to +\infty$.\par
Now, we estimate $I_2$. We set $K_R=B^+_{2R}\setminus B^+_{R}$.\\
Since $g$ has compact support, due to \eqref{86}, we get $|\nabla \psi |\leq \frac{B}{|x|^n}$ for all $x\in \R^n_+$ such that $|x|>2\,diam(supp\,g)$. So, applying H\"older's inequality,  we get:
\begin{align*}
|I_2|&\leq \sfrac{1}{R}\dm \nabla \psi\dm _{L^{\infty}(K_R)}\int_{B^+_{2R}}|\pi - \overline{\pi}_R|\,dx\leq \sfrac{C}{R}R\dm \nabla \psi\dm _{L^{\infty}(K_R)} \int_{B^+_{2R}}|\nabla \pi|\,dx  \\
&= C\dm \nabla \psi\dm _{L^{\infty}(K_R)}\int_{B^+_{2R}}|\nabla \pi| w^{-1}\,w\,dx \leq C\dm \nabla \psi\dm _{L^{\infty}(K_R)}\dm \nabla \pi \dm _{w,q} \Big[\int_{B^+_{2R}}\sfrac{1}{w^{\,q'}}\Big]^{\frac{1}{q'}}\\
&\leq C\dm \nabla \pi \dm _{L^q_w(\R^n_+)} R^{-\frac{n}{q}-\alpha}\,,
\end{align*}
which tends to zero as $R\to +\infty$.\\
So, we deduce that:
\begin{equation}
(\nabla \pi,g)=0,
\end{equation}
for all $g\in C^{\infty}_0(\R_+^n)$. Thus $f=0$, and this contradiction implies the thesis.
\end{proof}
We are in position to prove Helmholtz decomposition for $L^q_{w}(\R_+^n)$-spaces
\begin{theorem}\label{dec_helmotz_teorema}{\sl 
Let $q \in (1,+\infty)$. Then, there holds
\begin{equation}
L^q_{w}(\R_+^n)=J^q_{w}(\mathbb{R}_+^n)\oplus G^q_{w}(\mathbb{R}_+^n)\,,
\end{equation}
 that is for all $u \in L^q_{w}(\R_+^n)$, $u=v+\nabla \pi_u$, with the following integral identities:
\begin{gather}
(u,\nabla \pi)=(\nabla \pi_u ,\nabla \pi) \quad \text{for all}\, \, \nabla \pi \in G_{w'}^{q'}(\R_+^n)\, ,  \label{1)} \\
(v,\nabla \pi)=0 \quad \text{for all}\, \, \nabla \pi \in G_{w'}^{q'}(\R_+^n) \label{2)}\, ,
\end{gather}
and
\begin{equation} \label{3}
\dm v\dm _{L^q_w(\R^n_+)} + \dm  \nabla \pi \dm _{L^q_w(\R^n_+)} \leq C \dm u\dm _{L^q_w(\R^n_+)}\,
\end{equation}
with $C$ independent of $u$.}
\end{theorem}
\begin{proof}
Let us first consider a field $u \in C^{\infty}_0(\R_+^n)$. As we have done before, we denote by $\psi$ the solution of the problem \eqref{NP} and for such solution \eqref{stein} holds. We set:
\begin{equation}\label{setting}
v:= u- \nabla \psi\,.
\end{equation}
We claim that $(v,\nabla \pi)=0$ for all $\pi \in W^{1,q'}_{loc}(\R_+^n,w')$, with $\nabla \pi \in L_{w'}^{q'}(\R_+^n) $. Considering $h_R$, $\overline{\pi}_R$ and $R$ be as in the proof of Lemma \ref{lemma}, with the same arguments as before, we have:
\begin{align*}
(v,\nabla \pi)&=(v,h_R \nabla (\pi - \overline{\pi}_R) + (v,(1-h_R)\nabla \pi)=\\
&=-(\nabla \cdot v \, h_R,(\pi-\overline{\pi}_R)) - (v\cdot \nabla h_R,\pi-\overline{\pi}_R)+(v,(1-h_R)\nabla \pi)\,.
\end{align*}
Definition \eqref{setting} furnishes $\nabla \cdot  v=0$ and, in particular, $v\cdot \nu=0$. So, the first term on the right-hand side is zero. Considering $|(v, \nabla \pi)|$, it is enough to argue as in the case of Lemma \ref{lemma} to find that $(v, \nabla \pi)=0$. So, $v \in J_{w}^{q}(\R_+^n) $ and this implies \eqref{2)}. From \eqref{2)} and \eqref{setting}, we deduce \eqref{1)} and \eqref{3}.\par
Now, we consider $u \in L^{q}_{w}(\R^n_+)$. There exists a sequence $\{u_k\}\subset C^{\infty}_0(\R_+^n)$ converging to $u$ in $L^{q}_{w}$-norm. Thus:
\begin{equation*}
u_k=v_k+\nabla \psi_k\,,
\end{equation*}
with $v_k \in J^{q}_{w}(\R_+^n)$ and $\nabla \psi_k \in G^{q}_{w}(\R_+^n)$. By virtue of \eqref{3}, we get:
\begin{equation*}
\dm v_k-v_m\dm _{w,q}+\dm \nabla \psi_k-\nabla \psi_m\dm _{w,q}\leq C\dm u_k-u_m\dm _{w,q}\,.
\end{equation*}
In particular, $v_k$ converging to a field $v$ and $\nabla \psi_k$ converging to a field $h$ in $L^{q}_{w}$-norm. We claim that $v \in J^{q}_{w}(\R_+^n)$ and $h \in G^{q}_{w}(\R_+^n)$. As the matter of fact, we get:
\begin{equation*}
(v, \nabla p)=\lim_{k \to + \infty}(v_k, \nabla p)=0\,,
\end{equation*}
for all $p \in W^{1,q'}_{loc}(\R_+^n,w')$, with $\nabla p \in  L^{q'}_{w'}(\R_+^n)$, and considering $\varphi \in \mathscr{C}_0(\R_+^n)$, we have:
\begin{equation*}
(h,\varphi)=\lim_{k \to + \infty}(\nabla \psi_k, \varphi)=-\lim_{k \to + \infty}(\psi_k, \nabla\cdot \varphi)=0\,.
\end{equation*}
Thus, there exists  $\pi_u \in W^{1,q}_{loc}(\R_+^n,w)$ such that $h=\nabla \pi_u$. Moreover, applying Minkowski's inequality, we arrive at
\begin{align*}
\dm u-(v+\nabla \pi_u)\dm _{w,q}\leq& \dm u- u_k +u_k-(v+\nabla \pi_u)\dm _{w,q}\leq \\ &\dm u- u_k\dm _{w,q} + \dm v_k-v \dm _{w,q} +
\dm  \nabla \psi_k -\nabla \pi_u\dm _{w,q}\,,
\end{align*}
which tends to zero as $k\to +\infty$.\\
Then, we get
\begin{equation*}
u=v+\nabla \pi_u\,,
\end{equation*}
and
\begin{equation*}
\dm v\dm _{w,q}+ \dm \nabla \pi_u\dm _{q,\beta}=\lim_{k \to + \infty}\Big(\dm v_k\dm _{w,q}+ \dm \nabla \psi_k\dm _{w,q}\Big)\leq C \lim_{k \to + \infty}\dm u_k\dm _{w,q}= C\dm u\dm _{w,q}\,.
\end{equation*}
This completes the proof.
\end{proof}
\begin{lemma}{\sl 
Let $q\in (1,+\infty)$. Then, $J_{w}^q(\R_+^n)$ is the completion of $\mathscr{C}_0(\R_+^n)$ in $L_{w}^q$-norm.}
\end{lemma}
\bp See \cite{GPG}, \cite{SS} or \cite{HDS}. \ep
 {\bf Acknowledgment} -
We would like to thank Professor P. Maremonti for his valuable supervision and continuous support throughout this research. \\
The paper is
 performed under the
auspices of GNFM-INdAM.
\vskip0.1cm\noindent
 {\bf Declarations}
\vskip0.1cm\noindent
 {\bf Funding} -  The research activity of V. P. was partially supported by Università degli Studi della Campania “Luigi Vanvitelli”, D.R. 111/2024, within VISCOMATH project.
\vskip0.1cm\noindent
 {\bf Conflict of interest} - The authors have no conflicts of interest to declare that are relevant to the content of this article.


\end{document}